\documentclass[11pt,letterpaper]{amsart}
\usepackage{amsmath}
\usepackage{amssymb}

\usepackage{epsfig}
\usepackage[all, knot]{xy}
\xyoption{arc}

\newtheorem{thm}{Theorem}[section]
\newtheorem{df}[thm]{Definition}
\newtheorem{prop}[thm]{Proposition}
\newtheorem{cor}[thm]{Corollary}

\newtheorem{lem}[thm]{Lemma}
\newtheorem{ex}[thm]{Example}
\newtheorem{rem}[thm]{Remark}
\newtheorem*{reminder}{Theorem~1.2}
\newcommand{\Pic}{\operatorname{Pic}}

\newcommand{\Pre}{\operatorname{Preper}}

\newcommand{\pp}{\mathbb{P}}

\newcommand{\pn}{\mathbb{P}^n}

\newcommand{\prm}{\mathbb{P}^m}
\newcommand{\pni}{\mathbb{P}^{n_i}}

\newcommand{\pnk}{\mathbb{P}^{n_1} \times \cdots \times \mathbb{P}^{n_k}}
\newcommand{\ppn}{\mathbb{P}^{n_1} \times \cdots \times \mathbb{P}^{n_l}}

\begin{document}

\title[morphisms and heights]{Height estimates for dominant endomorphisms \\on projective varieties}

\author{Chong Gyu Lee}


\date{\today}


\address{Department of Mathematics, University of Illinois at Chicago, Chicago IL 60607, US}

\email{phiel@math.uic.edu}

\maketitle

\begin{abstract}

If $\phi$ is a polarizable endomorphism on a projective variety, then the Weil height
function gives a relation between the height of a point and the height
of its image under $\phi$.  In this paper, we generalize this result
to arbitrary dominant endomorphisms.  We define height expansion and
contraction coefficients for dominant morphisms, compare these to
Silverman's height expansion coefficient in \cite{S3}, and provide
several examples of dynamical systems on projective varieties.
\end{abstract}

\section{Introduction}

    A dynamical system consists of a set $S$ and a map $\phi :S \rightarrow S$ maps $S$ to $S$ itself. Thus, more structure $S$ has, more dynamical information we gain. When $S$ is a projective variety, it has the Weil height functions so that arithmetic dynamics gains lots of information from them. Moreover, if we have a special kind of morphism, then we have pleasant result: we say that $\phi$ is \emph{polarizable} if there is an ample divisor $D\in \Pic(S) \otimes_\mathbb{Z} \mathbb{R}$ such that $\phi^*D $ is linearly equivalent to $q\cdot D$ where $q$ is a positive real number. If $\phi$ is a polarizable defined over a number field $K$, then it satisfies the Northcott's property: we say \emph{$\phi$ satisfies the Northcott's property} if the following equality holds for some Weil height function $h_D$ corresponding an ample divisor $D$:
    \[
    h_D\left( \phi(P) \right) = q\cdot  h_D(P) + O(1) \quad \text{for all}~P\in S(\overline{K}).
    \]

    If $\phi$ is not polarizable, then it does not satisfy the Northcott's property. For example, an automorphism of infinite order on $K3$ surface is not polarizable so that we can't expect the above height inequality. However, we can still expect to find the relation between the height values of points $P$ and $\phi(P)$. We say that \emph{$\phi$ satisfies the weak Northcott's property} if there are a Weil height $h_D$ corresponding an ample divisor and two constants $C_1, C_2$ such that
    \[
    C_1 \cdot h_D\left( \phi(P) \right)  - O(1) \leq    h_D(P)  \leq C_2 \cdot h_D\left( \phi(P) \right)  +O(1).
    \]

    The main purpose of this paper is that every `dominant' endomorphism satisfies the weak Northcott's property. In section~$2$, Every dominant endomorphisms generates a map on the ample cone. This fact allows us to find the constants for the weak Northcott's property. (Well-definedness is guaranteed by Lemma~\ref{ample}.)
    \begin{df}\label{coefficients}
            Let $W$ be a projective variety and let $\phi : W \rightarrow W$ be a dominant morphism. We define \emph{the height expansion coefficient of $\phi$ for $D$}
                \[
                      \mu_1(\phi,D) := \sup\{ \alpha \in \mathbb{R} ~|~  \phi^*D - \alpha D~\text{is ample}\}
                \]
                and \emph{the height contraction coefficient of $\phi$ for $D$}
                \[
                      \mu_2(\phi,D) := \inf \{ \alpha \in \mathbb{R} ~|~ \alpha D - \phi^*D~\text{is ample}\}.
                \]
    \end{df}

    \begin{thm}\label{main}
        Let $W$ be a projective variety, let $\phi:W\rightarrow W$ be a dominant endomorphism defined over a number field $K$, let $D$ be an ample divisor on $W$ and $\mu_1= \mu_1(\phi,D) ,\mu_2= \mu_2(\phi,D)$ be the height expansion and contraction coefficients of $\phi$ for $D$. Then, for any $\epsilon >0$, there are constants $C_1, C_2$ satisfying
        \[
         \dfrac{1}{\mu_1 - \epsilon}h_D\bigl( \phi(P) \bigr) +C_1 \geq  h_D(P) \geq \dfrac{1}{\mu_2 + \epsilon}  h_D(P) - C_2
        \]
        for all $P\in W(\overline{K})$.
    \end{thm}

    Interestingly, we have Silverman's height expansion coefficient defined on \cite{S3}: a dominant endomorphism is clearly an example of a equidimensional dominant rational map. In section~3, we will show that they are the Silverman's height expansion coefficient is the same with $\mu_1$;
    \begin{prop}\label{Silverman}
    Let $\phi: W\rightarrow W$ be a dominant morphism defined over a number field, let $D$ be an ample divisor and let $\mu_1(\phi, D)$ be the height expansion coefficient of $\phi$. Then,
    \[
    \mu_1 (\phi,D) = \liminf_{h_D(P) \rightarrow \infty } \dfrac{h_D\bigl(\phi(P) \bigr)}{h_D(P)}.
    \]
    \end{prop}

    From now on, we will let $W$ be a projective variety, let $\phi : W \rightarrow W$ be a dominant endomorphism on $W$ defined over a number field $K$ and let $D$ be an ample divisor on $W$ unless state otherwise.

\par\noindent\emph{Acknowledgements}.\enspace
 The author would like to thank Joseph H. Silverman and Dan Abramovich for their helpful advice and comments.

\section{Dominant endomorphism and pull-backs of ample divisors}

    To satisfy the weak Northcott's property, $\phi$ should be at least quasi-finite: suppose not. Then we have a point $P$ whose inverse image is a subvariety $Y$. Thus, $h_D(P)$ is constant while $h_D(Q)$ goes to infinity on $Y$. Usually, a dominant morphism need not be quasi-finite. However, for endomorphism on a projective variety, `quasi-finiteness' condition is equivalent to `dominance' condition.

    \begin{df}
        Let $\psi:W \rightarrow V$ be a rational map. We say that \emph{$\psi$ is dominant} if $\overline{\psi(W)} = V$.
    \end{df}

    \begin{prop}
        Let $\phi :W \rightarrow W$ be an endomorphism. Then The followings are equivalent;
        \begin{enumerate}
            \item $\phi$ is dominant.
            \item $\phi$ is quasi-finite.
            \item $\phi$ is finite.
        \end{enumerate}
    \end{prop}

    \begin{proof}
        (1) $\Rightarrow$ (3) \quad Since $W$ is a projective variety, $W$ is compact and hence $\phi$ is surjective. Then, \cite[\S 4]{P}
        says that surjective holomorphic endomorphism on a projective variety is finite.

        (3) $\Rightarrow$ (1) \quad It is a property of finite morphism; if $\phi$ is not dominant, then $\phi$ is not quasi-finite and hence not finite.

        (2) $\Leftrightarrow$ (3) \quad \cite[\S 8.11.1]{G} says that $\phi$ is finite if $\phi$ is proper, locally of finite presentation and quasi-finite. Since $W$ is a projective variety, $\phi$ is automatically projective and hence proper and locally of finite presentation.
        Therefore, if $\phi$ is quasi-finite, then $\phi$ is finite.
    \end{proof}

    Let $\phi$ be defined over a number field. To study the Weil height function value of the image of some morphism $h_D\bigl(\phi(P) \bigr)$, it is essential to observe $\phi^*D$ because of the functorial property of the Weil height machine:
    \[
    h_D\bigl(\phi(P) \bigr) = h_{\phi^*D}(P) + O(1).
    \]
    If $\phi : W \rightarrow W$ is a polarizable, then, by definition, there is an ample divisor $E$ such that $q\cdot D \sim \phi^*E$, which implies that $\phi^*E$ is ample. It is also true for general dominant endomorphism because $\phi$ is quasi-finite.

    \begin{prop}
        Let $\phi :W \rightarrow W$ be a morphism. Then, the followings are equivalent;
        \begin{enumerate}
            \item $\phi$ is dominant.
            \item $\phi^*E$ is ample for some ample divisor $E$.
            \item $\phi^*E$ is ample for all ample divisors $E$.
        \end{enumerate}
    \end{prop}
    \begin{proof}
        (1) $\Rightarrow$ (3) \quad Suppose that $\phi^*E$ is not ample for an ample divisor $E$. Then, By Kleiman's criterion, there is a pseudo-effective  1-cycle $C$ (limit of effective cycle) such that $C \cdot \phi^*E \leq 0$. More precisely, since $E$ is ample and hence numerically effective, $\phi^*E$ is also numerically effective and hence $C \cdot \phi^* E= 0$. Because of projection formula for intersection, we have
        \[
        \phi_* C \cdot E = C \cdot \phi^*E = 0.
        \]
        If $\phi_*C$ is pseudo-effective 1-cycle, then $\phi_*C \cdot E >0$ because of Kleiman's Criterion again. It is contradiction so that $\phi(C)$ should be a zero-cycle and hence numerically equivalent to finite sum of points. However, $\phi$ is dominant and hence is quasi-finite. Therefore, the preimage of a finite set of points is a finite set of points again, so we again have a contradiction.

        (3) $\Rightarrow$ (2) \quad It is trivial.

        (2) $\Rightarrow$ (3) \quad Let $D$ be an ample divisor on $W$ such that $\phi^*D$ is ample. Suppose that there is an ample divisor $E$ such that $\phi^*E$ is not ample. Then, by Nakai-Moishezon Criterion, there is an integral subvariety $Y\subset W$ of dimension $r$ such that
            \[
            (\phi^*E)^r \cdot Y \leq 0.
            \]
            Then, by the projection formula for intersection, we have
            \[
            E^r \cdot \phi_*Y = \phi^*(E^r) \cdot Y = (\phi^*E)^r \cdot Y \leq 0.
            \]
            If $\phi_* Y$ is a subvariety of dimension $r$, then it is contradiction because $E$ is an ample divisor. So, $\phi_*Y$ should be of dimension $r' < r$. However, since $\phi^*D$ is ample,
            \[
            0< (\phi^*D)^r \cdot Y  = D^r \cdot \phi_*Y = 0
            \]
            and hence it is also contradiction. Therefore, $\phi^*E$ is also ample.

        (3) $\Rightarrow$ (1) \quad If $\phi$ is not dominant, then $\dim \phi(W) < \dim W$. So, there is a subvariety $V \subset W$ such that $\phi(V) = Q \in W$.
        Therefore, for an ample divisor $E$,
        \[
        h_{\phi^*E}(P) = h_E\bigl( \phi(P) \bigr) +O(1) = h_E(Q) + O(1) \quad \text{for all}~P\in V.
        \]
        Thus, the height corresponding $\phi^*E$ is bounded on a variety $V$ and hence $\phi^*E$ is not ample.
    \end{proof}

\section{The height expansion and contraction constant}

    In this section, we will define the height expansion and contraction coefficients and will build the height inequality. It starts from a basic property of ample divisors:
    \begin{lem}\label{ample}
        Let $W$ be a projective variety and $D_1, D_2$ be ample divisors on $W$. Then, there is a positive constant $\alpha$ such that
        $\alpha D_1 - D_2$ is ample again.
    \end{lem}
    \begin{proof}
        \cite[Theorem A.3.2.3]{SH} or \cite{H}.
    \end{proof}
    \noindent The Lemma~\ref{ample} guarantees the well-definedness of Definition~\ref{coefficients} since $\{ \alpha \in \mathbb{R} ~|~  \phi^*D - \alpha D~\text{is ample}\}$ is not an empty set. Once well defined, the height expansion and contraction coefficients will provide the weak Northcott's property;

    \begin{reminder}
        Let $\phi:W\rightarrow W$ be a dominant endomorphism defined over a number field $K$, let $D$ be an ample divisor on $W$ and $\mu_1= \mu_1(\phi,D) ,\mu_2= \mu_2(\phi,D)$ be the height expansion and contraction coefficients of $\phi$ for $D$. Then, for any $\epsilon >0$, there are constants $C_1, C_2$ satisfying
        \[
         \dfrac{1}{\mu_1 - \epsilon}h_D\bigl( \phi(P) \bigr) +C_1 \geq  h_D(P) \geq \dfrac{1}{\mu_2 + \epsilon}  h_D(P) - C_2
        \]
        for all $P\in W(\overline{K})$.
    \end{reminder}

    \begin{proof}
    Then, for any $\epsilon >0$, both $E_1 =  \phi^*D - (\mu_1 - \epsilon) D$ and $E_2 = (\mu_2 + \epsilon)D - \phi^*D$ are ample. Thus,
    $h_{E_1}$ and $h_{E_2}$ are bounded below. Therefore,
    \[
     h_D\bigl( \phi(P) \bigr) - (\mu_1 - \epsilon) h_D(P) = h_{E_1}(P) + O(1) > O(1)
    \]
    and
    \[
    (\mu_2 + \epsilon) h_D(P) - h_D\bigl( \phi(P) \bigr)  = h_{E_2}(P) + O(1) > O(1).
    \]
    Finally,
    \[
    \dfrac{1}{\mu_1 - \epsilon}h_D\bigl( \phi(P) \bigr) +C_1 \geq  h_D(P) \geq \dfrac{1}{\mu_2 + \epsilon}  h_D(P) - C_2.
    \]
    \end{proof}

    \begin{rem}
    We may expect the following inequality:
        \[
         \dfrac{1}{\mu_1 }h_D\bigl( \phi(P) \bigr) +C_1 \geq  h_D(P) \geq \dfrac{1}{\mu_2 }  h_D(P) - C_2.
        \]
        Unfortunately, it may not be true because $\phi^*D - \mu_1D$ and $\mu_2 \phi^*D - D$ are just numerically effective divisors so that the Weil heights corresponding to those divisors may not be bounded below on entire $W$.

        For example, Let $W$ be an elliptic curve and let $\phi = [N]$. Choose a point $P$ and let $Q=[N](P)$. Then, the divisor $q(P) - (Q)$ is ample if and only if $q>1$ and hence $\mu_1([N],(P)) = 1$. However, $\widehat{h}_Q (R) - \widehat{h}_P (R)= \widehat{h}_{Q-P}(R) = 2 \langle Q-P,R \rangle$ may go to $-\infty$.
    \end{rem}

    \begin{ex}
        Suppose that $\phi$ is a polarizable morphism with respect to an ample divisor $D$: \[\phi^* D\sim q\cdot D.\] Then, $\mu_1(\phi,D) = \mu_2(\phi,D) = q$ and hence it satisfies the Northcott's property.
    \end{ex}

    \begin{ex}
        Let $V\subset \pp^2 \times \pp^2$ be a $K3$-surface and let $\imath_1$, $\imath_2$ be involutions on $V$. Let $D_1,D_2$ be pullbacks of $H \times \pp^2$ and $\pp^2 \times H$ and $E_+ = -D_1 + \beta D_2$, $E_- = D_2 + \beta^{-1} D_1$ where $\beta = 2+ \sqrt{3}$. Then, divisor $D = aE_+ + bE_-$ is ample if and only if $a,b>0$.

        Then, $\imath_1^*(aE_+ + bE_-) = \beta (aE_-) + \beta^{-1} (bE_+)$. Thus,
        \begin{eqnarray*}
        \mu_1(\imath_1,E_+ + E_-) &=& \sup\{ \alpha ~|~ \beta^{-1} -\alpha  > 0, \beta  - \alpha >0 \} \\
                                    &=& \min \left( \beta^{-1}  \beta \right)\\
                                    &=& \beta^{-1}.
        \end{eqnarray*}

        Let $\phi = \imath_2 \circ \imath_1$. Then, it is dominant because
        \[
        \phi^*(aE_+ + bE_-) = \imath_1^* \bigl( \imath_2^*(aE_+ + bE_-) \bigr) = \imath_1^*(\beta aE_- + \beta^{-1} bE_+) = \beta^{-2} aE_+ + \beta^2 bE_-.
        \]
        Thus,
        \[
        \phi^*(aE_+ + bE_-) - \alpha (aE_+ + bE_-) = a( \beta^{-2}-\alpha) E_+ + b(\beta^2-\alpha) E_-.
        \]
        Therefore, $\mu_1(\phi,aE_+ + bE_-) = \beta^{-2}$ and hence $\mu(\phi) = \beta^{-2}$. Similarly, $\mu_2(\phi,aE_++bE_-) = \beta^{2}$.
    \end{ex}

    \begin{ex}
        Let $V\subset \pp^1 \times \pp^1 \times \pp^1$ be a generic hypersurface of tridegrees $(2,2,2)$. Let $\imath_1$, $\imath_2$ and $\imath_3$ be involutions on $V$. Then, the ample cone is the light cone
        \[
        \mathcal{L}^+ = \{E \in \Pic(V) ~|~ E^2>0, E\cdot D_0>0\}
        \]
        where $D_0$ is arbitrary ample divisor.
        Let $E_i$ be pullbacks of hyperplane $H_i$ of $i$-th component. Since the Picard number of $V$ is three, $\{E_1, E_2, E_3\}$ is a generator of $\Pic(V)$. Moreover, $E_a = E_1+ E_2 +E_3$ is very ample divisor corresponding Segre embedding and the intersection number of $\{E_1, E_2, E_3\}$ is
        \[
        \left(
            \begin{matrix}{ccc}
                0 &2&2 \\
                2&0&2\\
                0&2&2
            \end{matrix}
        \right).
        \]
        Therefore, the ample cone is described with the coefficient:
        \[
        \{\sum a_i E_i ~|~ \sum_{i\neq j} a_i a_j >0, \sum a_i >0\}.
        \]

        Then, $\imath_1^*D = -a_1E_1 + (2a_1+a_2)E_2 + (2a_1+a_3) E_3$. Thus,

        \begin{eqnarray*}
        \mu_1(\imath_1,E_a) &=& \sup\{ \alpha ~|~ (-1-\alpha)E_1 +  (3-\alpha) E_2 + (3-\alpha)E_3 :\text{ample} \} \\
                                    &=& \sup\{ \alpha ~|~ (5-3\alpha)>0, (\alpha-3)(3\alpha-1)>0 \} \\
                                    &=& \dfrac{1}{3}.
        \end{eqnarray*}
        and
        \begin{eqnarray*}
        \mu_2(\imath_1,E_a) &=& \inf\{ \alpha ~|~ (\alpha+1)E_1 +  (\alpha-3) E_2 + (\alpha-3)E_3 :\text{ample} \} \\
                                    &=& \inf\{ \alpha ~|~ (3\alpha-5)>0, (\alpha-3)(3\alpha-1)>0 \} \\
                                    &=& 3.
        \end{eqnarray*}

        Let $\phi_{1,2} = \imath_2 \circ \imath_1$. Then, it is dominant because
        \[
        \phi_{1,2}^*E_a = \imath_1^* \bigl( \imath_2^*E_a \bigr) = \imath_1^*(3E_1 - E_2 +3E_3) = -3E_1 + 5E_2 + 9E_3.
        \]
        Thus,

        \begin{eqnarray*}
        \mu_1(\phi_{1,2},E_a) &=& \sup\{ \alpha ~|~ (-3-\alpha)E_1 +  (5-\alpha) E_2 + (9-\alpha)E_3 :\text{ample} \} \\
                                    &=& \sup\{ \alpha ~|~ (11-3\alpha)>0, 3\alpha^2-22\alpha+3>0 \} \\
                                    &=& \dfrac{11-\sqrt{112}}{3}.
        \end{eqnarray*}
    \end{ex}

    \begin{ex}
        Let $\mathbb{X} : = \pp^{n_1} \times \pp^{n_k}$ where $n_i < n_{i+1}$ and let $\phi$ be a dominant endomorphism of $\mathbb{X}$ defined over a number field. Then, by Appendix~A, $\phi = (\phi_1, \cdots, \phi_k)$ where $\phi_i : \pp^{n_i} \rightarrow \pp^{n_i}$ is a morphism on projective space. Let $\pi_i : \mathbb{X} \rightarrow \pp^{n_i}$ be a projection map, let $\iota_i : \pp^{n_i} \rightarrow \mathbb{X}$ be a closed embedding map and let $E_i = \pi_i^*H_i$ where $H_i$ is a hyperplane of $\pp^{n_i}$. Then, a divisor $D = \sum_i=1^k a_i E_i$ is ample if and only if $a_i>0$ for all $i$. Furthermore, $\phi^*E_i = \deg \phi_i \cdot E_i$ and hence
        \[
        \mu_1(\phi,D) = \min \deg \phi_i \quad \mu_2(\phi,D) =\max \deg \phi_i.
        \]
    \end{ex}

\section{Silverman's height expansion coefficient}

    Silverman \cite{S3} introduced the height expansion coefficient for equidimensional dominant rational maps;
     \begin{df}\label{SC}
            Let $\psi :W \dashrightarrow V$ be a dominant rational map between quasiprojective varieties with the same dimension, all defined over $\overline{\mathbb{Q}}$. Fix height functions $h_{D_V}$ and $h_{D_W}$ on $V$ and $W$ respectively, corresponding to ample divisors $D_V$ and $D_W$. \emph{The height expansion coefficient of $\psi$} (relative to chosen ample divisors $D_V$ and $D_W$) is the quantity
            \[
            \mu' (\psi,D_W,D_V) = \sup_{\emptyset \neq U \subset W} \liminf_{P\in U(\overline{\mathbb{Q}})} \dfrac{h_{D_V} \bigl( \psi(P) \bigr) }{h_{D_W}(P)},
            \]
            where the sup is over all nonempty Zariski dense open subsets of $W$.
     \end{df}

     Then, the following theorem shows the relation between Definition~\ref{coefficients} and Definition~\ref{SC}

    \begin{proof}[Proof of Proposition~\ref{Silverman}]
     For dominant endomorphism $\phi: W\rightarrow W$, $\phi$ is defined on entire $W$. Thus, the supremum comes from the biggest open set of $W$, which is $W$ itself:
     \[
     \mu'(\phi,D,D) = \sup_{\emptyset \neq U \subset W} \liminf_{h_D(P)\rightarrow\infty} \dfrac{h_D\bigl(\phi(P)\bigr)}{h_D(P)} = \liminf_{h_D(P)\rightarrow\infty} \dfrac{h_D\bigl(\phi(P)\bigr)}{h_D(P)}.
     \]

        Let $\mu_1=\mu_1(\phi,D)$ and $\epsilon >0$ be any positive number. Then, there is a $\delta \in [0,\epsilon]$ such that $\phi^*D - (\mu_1-\delta) D$ is ample. Thus,
        \[
        h_{\phi^*D}(P) - (\mu_1-\delta) h_D(P) \geq O(1).
        \]
        Therefore,
        \begin{eqnarray}\label{ineq1}
        \dfrac{h_{\phi^*D}(P) -O(1)}{ h_D(P)} \geq \mu_1 -\delta \quad \text{and} \quad \liminf_{h_D(P)\rightarrow\infty}\dfrac{h_{\phi^*D}(P)}{ h_D(P)} \geq \mu_1-\delta \geq \mu_1 -\epsilon.
        \end{eqnarray}

        On the other hand, let $E = \phi^*D - (\mu_1+\epsilon)D$. Then, there is an irreducible curve $C$ such that $E\cdot C <0$; otherwise, then $E$ is a numerically effective divisor so that $E + \dfrac{\epsilon}{2}D$ is ample. But, it contradicts to the definition of $\mu_1$.

        Then, we have
        \[
        \lim_{\substack{h_D(P)\rightarrow \infty \\ P\in C}}\dfrac{h_E(P)}{h_D(P)} = \dfrac{E\cdot C}{D\cdot C} <0
        \]
        and hence
        \[
        \liminf_{h_D(P)\rightarrow \infty }\dfrac{h_E(P)}{h_D(P)} \leq \lim_{\substack{h_D(P)\rightarrow \infty \\ P\in C}}\dfrac{h_E(P)}{h_D(P)}<0.
        \]
        So,
        \begin{eqnarray}\label{ineq2}
        \liminf_{h_D(P)\rightarrow \infty }\dfrac{h_{\phi^*D}(P)}{h_D(P)} < \liminf_{h_D(P)\rightarrow \infty }\dfrac{h_{(\mu_1+\epsilon)D}(P)}{h_D(P)} = \mu_1+\epsilon.
        \end{eqnarray}

        Combine (\ref{ineq1}) and (\ref{ineq2}) and get
        \[
        \mu_1 - \epsilon  \leq \liminf_{h_D(P)\rightarrow \infty }\dfrac{h_D\bigl(\phi(P)\bigr)}{h_D(P)} \leq \mu_1+\epsilon
        \]
        for any $\epsilon>0$. Therefore, we get the desired result.
    \end{proof}

\section{applications}
    \subsection{arithmetic dynamics}
    The height expansion coefficient has an application in arithmetic dynamics. We know that $\Pre(\phi)$ is of bounded height when $\phi$ is polarizable with $q>1$. Recall that $q = \mu_1(\phi,D)$. Thus, it is not weird to expect the similar result for dominant endomorphism with the height expansion coefficient.

    \begin{df}
            Let $\phi : W(K) \rightarrow W(K)$ be a dominant morphism defined over a number field $K$. We define \emph{the global height expansion coefficient of $\phi$}:
        \[
        \mu (\phi) = \sup_{D: ~ ample} \mu_1(\phi,D).
        \]
    \end{df}

    \begin{thm}
        Let $\phi : W \rightarrow W$ be a dominant endomorphism and $E$ be an ample divisor. Suppose that
        the global height expansion coefficient $\mu (\phi) >1$. Then, the set of preperiodic points is of bounded height by $h_E$.
    \end{thm}
    \begin{proof}
        Let $\mu (\phi) > 1$. Then, there is an ample divisor $D$ such that $\mu_1(\phi,D) > 1$. Suppose that $\epsilon = \frac{\mu_1(\phi,D) -1}{2}$. Then,
        \[
        \dfrac{1}{\mu_1(\phi,D) - \epsilon}h_D\bigl( \phi(P) \bigr) = \dfrac{1}{1 + \epsilon}h_D\bigl( \phi(P) \bigr)  \geq  h_D(P) - C.
        \]
        By telescoping sum, we have
        \[
        \lim_{n\rightarrow \infty} \left( \dfrac{1}{1 + \epsilon} \right)^n h_D\bigl( \phi^n(P) \bigr)  \geq  h_D(P) - \dfrac{1}{1 - \frac{1}{1 + \epsilon}} C.
        \]
        Therefore, if $P \in \Pre(\phi)$, then the left hand side goes to zero so that $h_D(P)$ is bounded.

        Moreover, if $E$ is another ample divisor then Lemma~\ref{ample} says that $\alpha \cdot D - E$ is ample for sufficiently large $\alpha>0$. Since the Weil height corresponding the ample divisor is bounded below and hence
    \[ \alpha \cdot h_D(P)  + O(1) > h_E(P)
    \]
     for all $P\in W$. Therefore, $h_E\bigl( Pre(\phi) \bigr)$ is also bounded.
    \end{proof}

    \begin{ex}\label{domex}
        Consider the very first example; let $f_i : \pp^n \rightarrow \pp^n$ be a morphism of degree $d_i>1$. Then, a morphism
        \[
        \phi = \prod f_i : \left( \pp^n\right)^m \rightarrow  \left( \pp^n \right)^m
        \]`
        is a dominant morphism of $\mu(\phi) = \min d_i>1$. Thus, $\Pre(\phi)$ is a set of bounded height. Precise calculation appears on Appendix~A.
    \end{ex}

    \subsection{Seshadri Constant}

    The height expansion coefficient has a relation with the Seshadri constant. Demailly \cite{Dem} defined the Seshadri constant.

    \begin{df}
        Let $Y$ be a closed subscheme of $X$ whose underlying subvariety is of codimension $r>1$, let $\widetilde{X}$ be a blowup of $X$ along $Y$ and let $L$ be a numerically effective divisor of $X$. Then, we define the \emph{generalized Seshadri constant}
        \[
        \epsilon(L, Y) = \sup \{\alpha ~|~  \pi^*L - \alpha E :\text{ numerically effective}\}.
        \]
        Similarly, we define the $s$-invarinat
        \[
        s_L(Y) =  \min \{s ~|~  s\cdot \pi^*L - E :\text{ numerically effective}\}.
        \]
    \end{df}

    \begin{thm}
        Let $\phi :W \rightarrow W$ be a dominant morphism and let $D$ be a ample divisor. Then,
        \[
        \epsilon(\phi^*D , D) \geq \mu_1(\phi, D).
        \]
    \end{thm}
    \begin{proof}
        The ample cone of $V$ is a subcone of the nef cone and hence
        \begin{eqnarray*}
        \mu_1(\phi, D) & = &
         \sup \{ \alpha ~|~ \widetilde{\phi}^*D - \alpha D ~\text{is ample.} \} \\
          &=&  \sup \{ \alpha ~|~ \widetilde{\phi}^*D - \alpha D ~\text{is numerically effective.} \}\\
          &=& \epsilon(\phi^*D , D) .
        \end{eqnarray*}
    \end{proof}

\appendix
\section{Example: dominant morphisms on $\mathbb{X}=\pnk$}

    In this section, we will show that dominant endomorphisms on $\mathbb{X}=\pnk$ is a block diagonal one. So, we only have to treat Example~\ref{domex} in the view of arithmetic dynamics because for any dominant endomorphism $\phi$ on $\mathbb{X}$, there is a integer $N$ such that $\phi^N$ is a Cartesian product of endomorphisms $\psi_i : \pp^{n_i} \rightarrow \pp^{n_i}$.

\subsection{Basic notations for morphisms on $\mathbb{X}$}

        Let $H_i$ be a hyperplane of $\pni$ which generate $\Pic(\pni)$ and $\pi_i : \ppn \rightarrow \pni$ be a $i$-th projection map. Let $E_i = \pi_i^*H_i$. Then,  $\Pic(\mathbb{X})= \langle E_1 , \cdots, E_k \rangle$. Let $\mathfrak{X} = \Pic(\mathbb{X})\otimes \mathbb{R} = \sum \mathbb{R}E_i$ and consider $\phi^*$ as matrix with basis $\{ E_1 , \cdots, E_k \}$.
    \begin{lem}
        Let $D= \sum a_i E_i \in \Pic(\mathbb{X})$. Then $D$ is ample if and only if $a_i > 0$ for all $i$.
    \end{lem}
    \begin{proof}
        Clearly $D_0 = \sum_{i=1}^k E_i$ is ample; consider the Segre embedding
        \[
        \tau : \mathbb{X} \rightarrow \pp^N.
        \]
        Then, $D_0 = \tau^*H_{\pp^N}$ where $H_{\pp^N}$ is a hyperplane on $\pp^N$.

        For general $D$, we can find $\beta>0$ such that
        \[
        D = \beta \cdot D_0 + \sum \gamma_i E_i \quad \text{where}~\gamma_i \geq 0.
        \]
        Then, $D_0$ is ample and $\sum \gamma_i E_i$ is nef. Thus the sum of these two divisors is ample.
    \end{proof}

\subsection{Morphisms  $f : \mathbb{X} \rightarrow \prm$}

   In this subsection, we will study the morphism $f : \mathbb{X}
\rightarrow \prm$ which will be a component of endomorphism $\phi$ on
$\mathbb{X}$.
    For the convenience, assume that it's sorted by dimension: $n_i
\leq n_{i+1}$.

   \begin{lem}\label{zero} Let $f: \mathbb{X} \rightarrow \pp^m$ be a
morphism and
   \[
   f^* : \Pic(\pp^m) \otimes \mathbb{R} = \mathbb{R} \rightarrow
\Pic(\mathbb{X}) = \mathbb{R}^k, \quad 1 \mapsto (d_1, \cdots, d_k)
   \]
   Suppose $\sum_{i=1}^k (n_i +1) > m+1 $. Then, there is an index $j$ such
that $d_j=0$.
   \end{lem}
   \begin{proof}
       Let $A^*X$ be a Chow ring of a projective variety $X$. Then, \cite[Example8.3.4]{F} says that
       \[
       A^*X \otimes A^*\pp^l \simeq A^*(X \times \pp^l).
       \]
       Therefore,
       \[
       A^r(\mathbb{X}) = \bigotimes_{\sum r_i = r} A^{r_i}(\pp^{n_i}).
       \]

       Let $H$ be a hyperplane on $\pp^m$. Because $f^*$ is ring
homomorphism on the Chow ring and $f^*H \sim \sum_{i=1}^k d_i E_i$, $\bigl(
f^*H \bigr)^{m+1} = \bigl( \sum_{i=1}^k  d_i E_i \bigr)^{m+1}$ where
$H^{m+1}$ is $(m+1)$-th self intersection of $H$. Furthermore,
$H^{m+1} =0$ so that
       \[
       0 = \bigl( \phi^*H \bigr)^{m+1} = \left( \sum_{i=1}^k d_i E_i
\right)^{m+1} = \sum_I C_I E_1^{\alpha_{I1}} \cdots E_k^{\alpha_{Ik}}
    \]
    where $C_I =
        \left(
        \begin{array}{c} \alpha_1 , \cdots, \alpha_k   \\m+1\end{array}
        \right) d_1^{\alpha_1} \cdots d_k^{\alpha_k}
       $ are positive integers.
       However, the assumption $\sum_{i=1}^k (n_i +1) > m+1 $ does not allow
middle parts to vanish; $E_1^{\alpha_1} \cdots E_k^{\alpha_k} \neq 0$
if $\alpha_i \leq n_i$ for all $i = 1, \cdots, k$. Furthermore, $E_1^{\alpha_1} \cdots E_k^{\alpha_k}$ are linearly independent.
Therefore, $d_j=0$
for some $j$.
   \end{proof}

   \begin{cor}
       Let
       \[
       f :\mathbb{X} \rightarrow \prm
       \] be a morphism. Suppose $f^* =  (d_1, \cdots, d_k)$ $(f^*H = \sum_{i=1}^k d_i E_i)$. Then,
       $d_i = 0$ for all $i$ satisfying $n_i > m$.
   \end{cor}
   \begin{proof}
       Let $\iota_i : \pp^{n_1} \rightarrow \mathfrak{X}$ be a closed
embedding and let $H$ be a hyperplane on $\prm$. Then, $f \circ
\iota_i : \pp^{n_i} \rightarrow \pp^m$
       is a morphism such that $( f \circ \iota_i )^* H = d_i
H_{\pp^{n_i}}$. If $n_i > m$, then $d_i=0$ because of the
Lemma~\ref{zero}.
   \end{proof}

   \begin{thm}\label{forget}
       Let
       \[
       f : \mathbb{X} \rightarrow \pp^m
       \]
       be a morphism. If $\sum_{i=1}^k (n_i +1) > m+1$, then we can forget a
factor of $\mathbb{X} = \pnk $; there is a map
       \[
       g: \mathbb{X}' = \prod_{j \in J } \pp^{n_j}  \rightarrow \pp^m
       \]
       where $J \varsubsetneq \{1 ,\cdots, k\}$ such that
       \[
       \xymatrix
       {
       \mathbb{X} \ar[dr]^{f} \ar[d]_{\pi} &  \\
       \mathbb{X}' \ar[r]_{g} & \prm
       }
       \]
        Moreover, we can claim $\sum_{j\in J} (n_j +1) \leq m+1$.
   \end{thm}

   \begin{proof}
       Suppose the assumption is true. Then, by Lemma~\ref{zero},
$d_j=0$ for some $i$. Thus, we may assume that $\phi$ is a constant
morphism in terms of variables $X_j$ on $\pp^{n_j}$. Therefore, we can
consider
       \[
       g : \left( \pp^{n_1} \times \cdots \widehat{\pp^{n_j}} \cdots
\times  \pp^{n_k} \right)  \rightarrow \pp^{m}.
       \]
       If $\sum_{i\neq j} (n_1 +1) > m+1$, then we can apply
Lemma~\ref{zero} again.
   \end{proof}

\subsection{Endomorphisms on $\mathbb{Y}=(\pn)^l $}

    In the next subsection, we will have the dominant endomorphism on
$\mathbb{X}$ is a product of endomorphism on $(\pn)^l$. Thus, we will
check the endomorphism on $(\pn)^l$ to prepare the final result.

    Let $\mathbb{Y}=(\pn)^l $ and $\mathfrak{Y} = \Pic(\mathbb{Y})
\otimes \mathbb{R} = \mathbb{R}E_1 \oplus \cdots \oplus
\mathbb{R}E_l$.

   \begin{lem}\label{one}
       Let \[\psi(P) = :
\mathbb{Y}\longrightarrow \mathbb{Y}\] be a morphism. Then $\psi$
only depends on one of $\pn$.
   \end{lem}

       \begin{proof}
        Let $\psi^* = (d_1, \cdots, d_l)$. Then, since $n+1 < 2n+2$,
Theorem~\ref{forget} tells that exactly one of $d_i$ can be nonzero.
       \end{proof}

    \begin{cor}\label{matrix}
       Let $\psi : \mathbb{Y} \rightarrow \mathbb{Y}$ be a dominant
endomorphism. Then
       \[
       \psi^* : \mathfrak{Y} \rightarrow \mathfrak{Y}
       \]
       is a $l\times l$-matrix such that there is only one nonzero
element on each row;
       \[
       \psi^* = \left(
                       \begin{array}{c}
                                 d_1\mathbf{e}_{\sigma(1)}\\
                                  \vdots\\
                                  d_l\mathbf{e}_{\sigma(l)}
                       \end{array}
                \right)
       \]
       where $\mathbf{e}_j$ is $j$-th elementary row vector and
$\sigma \in S_l$ is a permutation map defined by $\psi^*E_i = d_{i}
E_{\sigma(i)}$.
    \end{cor}
    \begin{proof}
        Let $\psi = (\psi_1 , \cdots \psi_l)$. Lemma~\ref{one} says that $\psi_j^*$ is a row vector whose elements are zero except one.
        If $\psi_u^*$ and $\phi_v^*$ has nonzero element on the same column, then there is a zero column on $\psi^*$ and hence $\psi^*E$ is not ample for any ample divisor $E$. It contradicts to $\psi$ is dominant.
    \end{proof}

   \begin{cor}\label{diagonal}
       Let $\psi : \mathbb{Y} \rightarrow \mathbb{Y}$ be a dominant
endomorphism. Then, there is a natural number $N$ such that
$\left(\phi^N\right)^*$ is a diagonal matrix.
   \end{cor}

\subsection{\bf Endomorphisms on $\mathbb{X}$}

   \begin{lem}\label{triangular}
       Let $\phi$ is an endomorphism on $\mathbb{X}$. Then, $\phi^*$
is an upper block-triangular matrix.
       \[
       \phi^* =
           \left(
           \begin{array}{ccc}
           A_1 &\cdots&B\\
       O&\ddots &C\\
       O&\cdots&A_s
       \end{array}
       \right)
       \]
   \end{lem}

   \begin{proof}
    Let $n_u <n_v$. Suppose that $\phi = (\phi_1, \cdots, \phi_k)$
and $\phi_u^* = (d_{u1} , \cdots, d_{uk})$. Then,
    Consider a morphism
    \[
    \zeta = \phi_u \circ \iota_v :\pp^{n_v} \rightarrow \pp^{n_u}
    \]
     where $\iota_v : \pp^{n_v} \hookrightarrow \mathbb{X}$ is the
$v$-th closed embedding. Clearly, it is a morphism of degree $d_{uv}$.
But it should be a constant map because of Lemma~\ref{zero} so that
$d_{uv}=0$.
   \end{proof}

  For convenience, we will also use another expression $\mathbb{X} = \mathbb{Y}_1 \times
\cdots \times \mathbb{Y}_s$ where $\mathbb{Y}_j = \mathbb{P}^{m_j}
\times \cdots \times \mathbb{P}^{m_j}$ and $m_j < m_{j+1}$.

   \begin{thm}\label{morphism}
       Let $\phi$ is a dominant endomorphism on $\mathbb{X}$. Then,
$\phi^*$ is a block-diagonal matrix
       and hence
               \[\phi(P) =(\psi_1(P_1) , \cdots ,\psi_s(P_s)) \]
       where $\psi_j$ is a morphism on $\mathbb{Y}_j$. Furthermore,
each diagonal block is nonsingular and multiplication of a permutation
and a diagonal matrices.
   \end{thm}

   \begin{proof}
     Since $\phi^*$ is an upper block-triangular matrix by
Lemma~\ref{triangular}, it's enough to show all upper non-diagonal
block is zero.
     Let $n_u <n_v$. Suppose that $\phi_u^* =(d_{u1} , \cdots,
d_{uk})$ where $d_{uv} \neq 0$.
     Then, for any $w$ satisfying $n_w \geq n_v$,  $d_{uw} =0$
because $(\sum d_{ui}E_i)^{n_u+1}=0$ guarantees
     \[
     C d_{uv}^{n_u-1} d_{uv}^{n_v-n_u+1} E_u^{n_u} \cdot
E_v^{n_v-n_u+1} = 0
    \] while $E_u^{n_u} \cdot E_v^{n_v-n_u+1} \neq 0.$

     So, a diagonal block $A_j$ have a zero row and hence $\phi^*D$
can't be ample for all ample divisor $D$. Thus $\phi$ is not dominant and
it's a contradiction. Therefore $d_{uv} = 0.$
     Which means $\phi^*$ is a block-diagonal matrix:
     \[ \phi^* =
   \left(
   \begin{array}{ccc}
   \psi_1^* &O&O\\
   O&\ddots &O\\
   O&O&\psi_s^*
   \end{array}
   \right)
   \]
   which means
   \[
   \phi(P) = (\psi_1, \cdots , \psi_ s)
    \]
    where $\psi_j$ is a dominant endomorphism on $\mathbb{Y}_j$.
   \end{proof}

   \begin{cor}
       Let $\phi :\mathbb{X} \rightarrow \mathbb{X}$ be a dominant endomorphism. Then,
       for sufficiently large $N$, $\phi^N$ is a product of dominant endomorphism on $\pp^{n_i}$;
       \[
       \phi^N = \prod \phi_{N,i} : \pp^{n_i} \rightarrow \pp^{n_i}
       \]
       Moreover,
       \[
       \mu_1(\phi^N,D) = \mu(phi^N) = \min \deg \phi_{N,i}, \quad \mu_2(\phi^N,D 0 \max \deg \phi_{N,i}.
       \]
   \end{cor}

\end{document}